\documentclass[12pt]{amsart}

\usepackage{mathptmx}
\usepackage{amssymb}
\usepackage{eucal}
\usepackage{amsxtra}
\usepackage{verbatim}
\usepackage{enumerate}
\usepackage{enumitem}
\usepackage[english]{babel}
\usepackage{bm}
\usepackage{newtxmath}
\usepackage{setspace}
\usepackage[x11names]{xcolor}

\usepackage[T1]{fontenc}

\usepackage{graphicx}
\usepackage{mathrsfs}
\usepackage{geometry}
\usepackage{amscd,latexsym,amsthm,amsfonts,amssymb,amsmath,amsxtra}
 \usepackage[colorlinks, linkcolor=DodgerBlue3, citecolor=Red2]{hyperref}
\usepackage{enumerate}
\usepackage[all]{xy}%
\providecommand{\U}[1]{\protect\rule{.1in}{.1in}}
\RequirePackage{amsmath}
\RequirePackage{amssymb}
\usepackage{comment}

\usepackage{mathrsfs}
\usepackage{tikz-cd}

\newtheorem{theorem}{Theorem}[subsection]
\newtheorem{proposition}[theorem]{Proposition}

\newtheorem{definition}[theorem]{Definition}
\newtheorem{remark}[theorem]{Remark}

\newcommand{\Proj}{\underline{\mathrm{Proj}}}
\newcommand{\Spec}{\underline{\mathrm{Spec}}}
\newcommand{\Pro}{\mathrm{Proj}}
\newcommand{\Spe}{\mathrm{Spec}}
\newcommand{\Osh}{\mathcal{O}}
\newcommand{\Ish}{\mathcal{I}}
\newcommand{\Csh}{\mathcal{C}}
\renewcommand{\Lsh}{\mathcal{L}}

\newcommand{\Prm}{\mathbb{P}}
\newcommand{\Arm}{\mathbb{A}}
\newcommand{\Zrm}{\mathbb{Z}}

\begin{document}
\title[]
{Blow-Up Constructions and Applications to Segre Classes and Multidegree formulas}

\author{Kai Huang}
\date{\today}

\begin{abstract}
By comparing simultaneous multigraded blow-ups with suitable iterated blow-up construction, we establish a birational correspondence between the associated exceptional divisors. We further investigate blow-ups arising from rational maps to multiprojective spaces and derive intersection-theoretic formulas for the Chern classes of the resulting exceptional divisors and pullback of tautological line bundles.

These constructions have two main applications. First, they yield a proof of the product formula for Segre classes without any pure-dimensionality hypothesis. Second, they lead to a degree formula for arbitrary closed subschemes of multiprojective spaces, recovering and extending the classical formula of van der Waerden.
\end{abstract}
\maketitle
\section{Introduction}
The main goal of this paper is to study two geometric constructions arising from blow-ups and to illustrate their applications in intersection theory.

The first part of the paper is devoted to the comparison of simultaneous multigraded blow-ups with suitable iterated blow-up constructions. We establish a birational correspondence between these models and describe the resulting relations among their exceptional divisors. This comparison provides a geometric mechanism for transferring intersection-theoretic information between the two constructions.

The second part studies blow-ups associated with rational maps to multiprojective spaces. Starting from the identification of a blow-up with the scheme-theoretic closure of the graph of a rational map, we analyze the relations among the pullback of tautological line bundles and exceptional divisors and derive explicit formulas for their Chern classes.

As a first application, we revisit the product formula for Segre classes. Existing proofs generally require the schemes involved to be pure-dimensional, since this hypothesis is needed to apply the standard cycle-theoretic relations in intersection theory (such as \cite[Lemma 4.2, Proposition 4.2 and Example 4.2.4]{Ful98}). Using the birational correspondence constructed above, we show that the product formula is fundamentally a consequence of the geometry of blow-ups. More precisely, an appropriate exceptional divisor serves as a geometric bridge that reduces the problem from arbitrary closed embeddings to the case of Cartier divisors on a suitable resolution. This yields a proof of the formula that avoids any pure-dimensionality assumptions.

A second application concerns degree formulas for closed subschemes of multiprojective spaces. 
This result may be viewed as a modern reformulation and extension of the classical degree formula of van der Waerden \cite{vdW78}. The original argument was formulated in the language of classical algebraic geometry and relied essentially on the assumption that the varieties under consideration were integral (as   explained in Remark \ref{conterex}). By contrast, our proof is entirely intersection-theoretic, based on the formalism of Chow groups, blow-ups, and refined intersection products developed by Fulton. Besides providing a conceptual and functorial derivation of the formula, this approach naturally extends the result from varieties to arbitrary closed subschemes.

We follow the notation and conventions of \cite{Ful98} throughout the paper.
\section{Comparison of Two Blow-up Constructions}
Fix a base field $k$. Throughout this section, all schemes and morphisms are over $k$.
\subsection{Geometric Framework}
 Let $X_i \subset Y_i$ be closed subschemes defined by coherent ideal sheaves $\mathcal{I}_i$, for each $i=1,\ldots,n$. Let $Bl_i$ be the blow-up of $Y_i \times \mathbb{A}^1_{k}$ along the closed imbedding $\iota_i:X_i\hookrightarrow Y_i \times \mathbb{A}^1_{k}$ induced by $X_i \hookrightarrow Y_i$ and $X_i \to \Spe\ k\overset{\text{0}}{\hookrightarrow}\mathbb{A}^1_{k}$. Let $E_i\subset Bl_i$ be the exceptional divisor.

We begin with the following Cartesian diagram
\[
\begin{tikzcd}
E_i \arrow[r, "\tau_i"] \arrow[d, hook, "\rho_i"] & X_i \arrow[d, hook, "\iota_i"] \\
Bl_i \arrow[r, "\pi_i"] & Y_i \times \mathbb{A}^1_k.
\end{tikzcd}
\]
This diagram induces a Cartesian diagram on the products
\[
\begin{tikzcd}
E_1 \times \dots \times E_n \arrow[r, "\tau_1 \times \dots \times \tau_n"] \arrow[d, hook, "\rho_1 \times \dots \times \rho_n"'] & X_1 \times \dots \times X_n \arrow[d, hook, "\iota_1 \times \dots \times \iota_n"] \\
Bl_1 \times \dots \times Bl_n \arrow[r, "\pi_1 \times \dots \times \pi_n"] & Y_1 \times \mathbb{A}^1_k \times \dots \times Y_n \times \mathbb{A}^1_k
\end{tikzcd}
\]

Let $Bl'$ be the blow-up of $E_1 \times \dots \times E_n \hookrightarrow Bl_1 \times \dots \times Bl_n$ with exceptional divisor $E'$, while $Bl$ be the blow-up of $X_1 \times \dots \times X_n \hookrightarrow Y_1 \times \mathbb{A}^1_k \times \dots \times Y_n \times \mathbb{A}^1_k$ with exceptional divisor $E$.
\begin{proposition}\label{iterate}
There exist morphisms $\phi$ and $\varphi$ making the following diagram commute with Cartesian squares, such that $\varphi^*\Osh_{Bl}(1)\cong \Osh_{Bl'}(1)$, $\phi_*[E'] = [E]\in Z_*(E)$ and $\varphi_*[Bl'] = [Bl]\in Z_*(Bl)$:
\[
\begin{tikzcd}[row sep=1.8cm, column sep=1.1cm]
E' \arrow[rr, "\tau'"] \arrow[d, hook, "\rho'"] \arrow[rrrdd, bend right=12, "\phi"', pos=0.5] & & E_1 \times \dots \times E_n \arrow[rr] \arrow[d, hook] & & X_1 \times \dots \times X_n \arrow[d, hook] \\
Bl' \arrow[rr, "\pi'"] \arrow[rrrdd, bend right=12, "\varphi"', pos=0.5] & & Bl_1 \times \dots \times Bl_n \arrow[rr] & & Y_1 \times \mathbb{A}^1_k \times \dots \times Y_n \times \mathbb{A}^1_k \\
& & & E \arrow[d, hook, "\rho"] \arrow[uur, "\tau", pos=0.4] & \\
& & & Bl \arrow[uur, "\pi"', pos=0.6] &
\end{tikzcd}
\]
\end{proposition}
\begin{proof}
Since the inverse image ideal sheaf of $X_1 \times \dots \times X_n \hookrightarrow Y_1 \times \mathbb{A}^1_k \times \dots \times Y_n \times \mathbb{A}^1_k$ is the ideal sheaf of $E'\hookrightarrow Bl'$, which is isomorphic to the invertible sheaf $\Osh_{Bl'}(1)$, the universal property of the blow-up yields the morphism $\varphi$, while $\phi$ is induced by the universal property of the fiber product.

Note that if the blow-up center contains no generic points of the ambient space, the blow-up morphism induces a canonical identification between the generic points of the source and the target, with isomorphic local rings at each corresponding point. Therefore, the generic points of $Bl'$ correspond to those of $Bl_1 \times \dots \times Bl_n$ as $E_1 \times \dots \times E_n \hookrightarrow Bl_1 \times \dots \times Bl_n$ is a closed regular imbedding, which contains no generic points. Meanwhile, the generic points of $Bl$ correspond to those of $Y_1 \times \mathbb{A}^1_k \times \dots \times Y_n \times \mathbb{A}^1_k$.

For each $i = 1, \dots, n$, let $Z_i$ denote the closed subscheme of $Y_1 \times \mathbb{A}^1_k \times \dots \times Y_n \times \mathbb{A}^1_k$ obtained by replacing the $i$-th factor with $X_i$, namely,
\[ Z_i = Y_1 \times \mathbb{A}^1_k \times \dots \times X_i \times \dots \times Y_n \times \mathbb{A}^1_k. \]
Then, each $Z_i$ contains no generic points of the ambient space, and its scheme-theoretic inverse image, given explicitly by $Bl_1 \times \dots \times E_i \times \dots \times Bl_n$, contains no generic points of $Bl_1 \times \dots \times Bl_n$. Let $U$ be the open complement of the union $\bigcup_{i=1}^n Z_i$. By restricting the large diagram to $U$, the morphisms $\pi'$, $\pi$, and $\pi_1 \times \dots \times \pi_n$ all become isomorphisms, which implies that $\varphi$ is also an isomorphism over $U$. Since $U$ contains all generic points, from which it follows immediately that $\varphi_*[Bl'] = [Bl] \in Z_*(Bl)$.

The relationship between the exceptional divisors is more subtle. 
Note that the exceptional divisor is the projectivization of the normal cone. Let
\[
C_i=\Spec_{X_i}\Csh_i,
\qquad
\Csh_i=\bigoplus_{d\ge0}\Ish_i^d/\Ish_i^{d+1},
\]
be the normal cone of \(X_i\hookrightarrow Y_i\). Then
\[
E_i=\Proj_{X_i}\Csh_i[T_i],
\]
and
\[
E=
\Proj_{X_1\times\cdots\times X_n}
\bigl(
\Csh_1[T_1]|
\otimes\cdots\otimes
\Csh_n[T_n]|
\bigr),
\]
where \(\Csh_i[T_i]|\) denotes the pullback of \(\Csh_i[T_i]\) along \(X_1\times\cdots\times X_n\to X_i\). Since \(E_1\times\cdots\times E_n\hookrightarrow Bl_1\times\cdots\times Bl_n\) is a regular closed imbedding, \(E'\) is a \(\Prm^{n-1}\)-bundle over \(E_1\times\cdots\times E_n\). In particular, the projection
$E'\to E_1\times\cdots\times E_n$
induces a bijection on generic points. Since the open subscheme $
C_1\times\cdots\times C_n\subset E_1\times\cdots\times E_n$
contains all generic points of \(E_1\times\cdots\times E_n\), its inverse image contains all generic points of \(E'\). Moreover, as the \(\Prm^{n-1}\)-bundle is trivial over $
C_1\times\cdots\times C_n$,
its inverse image is identified with
\[
\Proj_{C_1\times\cdots\times C_n}
\bigl(
\Osh_{C_1\times\cdots\times C_n}[T_1,\ldots,T_n]
\bigr).
\]
Now we deduces a commutative diagram
\[
\begin{tikzcd}[column sep=0.7cm, row sep=large] 
\Proj_{C_1 \times \dots \times C_n} \left( \mathcal{O}_{C_1 \times \dots \times C_n} [T_1, \dots, T_n] \right) \arrow[r] \arrow[d] & 
E' \arrow[r, "\phi"] \arrow[d, "\tau'"] & 
E = \Proj_{X_1 \times \dots \times X_n} \left( \mathcal{C}_1[T_1]| \otimes \dots \otimes \mathcal{C}_n[T_n]| \right) \arrow[d, "\tau"] \\
C_1 \times \dots \times C_n \arrow[r] & 
E_1 \times \dots \times E_n \arrow[r, "\tau_1 \times \dots \times \tau_n"] & 
X_1 \times \dots \times X_n
\end{tikzcd}
\]
Consider the composite morphism
\[
\Proj_{C_1 \times \dots \times C_n}
\bigl(\mathcal{O}_{C_1 \times \dots \times C_n}[T_1,\dots,T_n]\bigr)
\longrightarrow
\Proj_{X_1 \times \dots \times X_n}
\bigl(\mathcal{C}_1[T_1]|\otimes \cdots \otimes \mathcal{C}_n[T_n]|\bigr).
\]
By construction of $\phi$, locally on affine charts (assuming all $X_i$ are affine), this morphism is given by
\[
\Pro\ (C_1\otimes_k\cdots\otimes_k C_n)[T_1,\dots,T_n]
\longrightarrow
\Pro\ \bigl(C_1[T_1]\otimes_k\cdots\otimes_k C_n[T_n]\bigr),
\]
and is induced by the homomorphism of graded rings
\[
C_1[T_1]\otimes_k\cdots\otimes_k C_n[T_n]
\longrightarrow
(C_1\otimes_k\cdots\otimes_k C_n)[T_1,\dots,T_n]
\]
sending \(T_i\) to \(T_i\) and homogeneous element
\(\alpha_i^{(d)}\in C_i\) of degree \(d\) to \(\alpha_i^{(d)}\cdot T_i^d\).
Since this morphism is an isomorphism over the dense open subscheme
\(D(T_1\cdots T_n)\), it induces a bijection between the generic points
of the source and the target, with isomorphic local rings at
corresponding points. Therefore, $
\phi_*[E']=[E]\in Z_*(E)$.
\end{proof}

\subsection{Application to Segre Classes}
Let $X\hookrightarrow Y$ be a closed imbedding. According to \cite[Section 4.1 and Section 4.2]{Ful98}, the Segre class of $X$ in $Y$ is the Segre class of the normal cone $C_X Y$ (corresponding to graded $\Osh_X$-algebra $\Csh_X Y$) of $X\hookrightarrow Y$:
\begin{equation}\label{2.9}
s(X,Y):=s(C_XY)
=q_*\left(
\sum_{i\ge 0}
c_1\bigl(\Osh(1)\bigr)^i
\cap
\bigl[\Proj_{X}(\Csh_X Y\oplus \mathbf{1})\bigr]
\right)\in A_*(X),
\end{equation}
where $q:\Proj_{X}(\Csh_X Y\oplus \mathbf{1})\to X$ is the projection morphism.

In case $X\hookrightarrow Y$ is a regular imbedding, $s(X,Y)$ is the cap product of the total inverse Chern class of the normal bundle with $[X]$.

\begin{remark}
Note that a well-known special case occurs when $X$ contains no generic
points of $Y$. In this situation,
\[
s(X,Y)
=
q'_*\left(
\sum_{i\ge 0}
c_1\bigl(\Osh(1)\bigr)^i
\cap
\bigl[\Proj_{X}(\Csh_XY)\bigr]
\right)
\in A_*(X).
\]
As an immediate consequence, $s(X,Y)=s(X,Y\times \Arm^1_k)$ for any closed imbedding $X\hookrightarrow Y$.

Indeed, let
\[
[C_XY]=\sum_i r_i[C_i]\in Z_*(C_XY)
\]
be the fundamental cycle of $C_XY$, and let $X_i$ denote the support of
$C_i$. Then
\[
[\Proj_X(\Csh_XY)]
=
\sum_i r_i[\Proj_{X_i}(C_i)]
\in Z_*\bigl(\Proj_X(\Csh_XY)\bigr).
\]
By \cite[Lemma~1.7.2]{Ful98},
\[
c_1\bigl(\Osh(1)\bigr)
\cap
\bigl[\Proj_X(\Csh_XY\oplus\mathbf{1})\bigr]
=
\bigl[\Proj_X(\Csh_XY)\bigr]\in Z_*\bigl(\Proj_X(\Csh_XY\oplus\mathbf{1})\bigr).
\]
Moreover, $
q_*\bigl[\Proj_X(\Csh_XY\oplus\mathbf{1})\bigr]=0$ 
for dimensional reasons. Therefore,
\[
q_*\left(
\sum_{i\ge0}
c_1\bigl(\Osh(1)\bigr)^i
\cap
\bigl[\Proj_X(\Csh_XY\oplus\mathbf{1})\bigr]
\right)
=
q'_*\left(
\sum_{i\ge0}
c_1\bigl(\Osh(1)\bigr)^i
\cap
\bigl[\Proj_X(\Csh_XY)\bigr]
\right),
\]
which yields the claimed formula.
\end{remark}
\begin{proposition}
Let $X_i\hookrightarrow Y_i$ be closed imbeddings for $i=1,...,n$. Then $s(X_1\times\cdots\times X_n,Y_1\times\cdots\times Y_n)=s(X_1,Y_1)\times\cdots\times s(X_n,Y_n)\in A_*(X_1\times\cdots\times X_n)$.
\end{proposition}
\begin{proof}
Note that if each $X_i\hookrightarrow Y_i$ is a regular imbedding, then so is
\[
X_1\times\cdots\times X_n
\hookrightarrow
Y_1\times\cdots\times Y_n .
\]
Furthermore, the normal bundle decomposes as
\[
N_{X_1\times\cdots\times X_n}
Y_1\times\cdots\times Y_n
\cong
\bigoplus_{i=1}^{n}
pr_i^{*}N_{X_i}Y_i .
\]
It follows that
\begin{align*}
s(X_1 \times \cdots \times X_n, Y_1 \times \cdots \times Y_n) &= \frac{1}{c(N_{X_1 \times \cdots \times X_n} Y_1 \times \cdots \times Y_n)} \cap [X_1 \times \cdots \times X_n] \\
&= \frac{1}{c(N_{X_1 \times \cdots \times X_n} Y_1 \times \cdots \times Y_n)}\cap ([X_1] \times \cdots \times [X_n] )\\
&= \left( \frac{1}{c(N_{X_1}Y_1)} \cap [X_1] \right) \times \cdots \times \left( \frac{1}{c(N_{X_n}Y_n)} \cap [X_n] \right) \\
&= s(X_1, Y_1) \times \cdots \times s(X_n, Y_n)
\end{align*}

For the general case, we use the construction given in Proposition~\ref{iterate}.
\begin{align*}
& s(X_1 \times \cdots \times X_n, Y_1 \times \cdots \times Y_n) = s(X_1 \times \cdots \times X_n, Y_1 \times \cdots \times Y_n \times \Arm_k^n) \\
& = \tau_* \left( \sum_{i \ge 0} c_1(\mathcal{O}(1))^i \cap [E] \right) = \tau_* \phi_* \left( \sum_{i \ge 0} c_1(\mathcal{O}(1))^i \cap [E'] \right) \quad \text{(Projection formula)}\\
& = (\tau_1 \times \cdots \times \tau_n)_* \circ \tau'_* \left( \sum_{i \ge 0} c_1(\mathcal{O}(1))^i \cap [E'] \right) = (\tau_1 \times \cdots \times \tau_n)_* s(E_1 \times \cdots \times E_n, Bl_1 \times \cdots \times Bl_n) \\
& = (\tau_1 \times \cdots \times \tau_n)_* \left( s(E_1, Bl_1) \times \cdots \times s(E_n, Bl_n) \right) = s(X_1, Y_1) \times \cdots \times s(X_n, Y_n)
\end{align*}

\end{proof}

\section{Isomorphism Between Graph Closure and Blow-up}
\subsection{Geometric Framework}
Let $X$ be a scheme, $\Lsh$ an invertible sheaf on $X$. Suppose we are given a finite collection of global sections $s_0, s_1, \dots, s_m \in H^0(X, \Lsh)$, not all of which vanish identically. These sections define a closed subscheme $V(s_0, \dots, s_m)$, whose ideal sheaf $\Ish \subset \Osh_X$ is the image of the evaluation map $\Osh_X^{m+1} \xrightarrow{(s_0, \dots, s_m)} \Lsh$, twisted by $\Lsh^{-1}$. 
Each $s_i$ lifts naturally to a section $s_i'\in H^0(X,\Ish\otimes\Lsh)$ via the canonical morphism $\Ish\otimes\Lsh \to \Lsh$.

On the open subscheme $U = X \setminus V(s_0, \dots, s_m)$, these sections define a morphism
\[ \phi: U \longrightarrow \Prm^m_{\Zrm} \]
by \cite[Theorem 7.1]{Har77}.

\begin{definition}
We define two schemes over $X$:
\begin{enumerate}
    \item \textbf{The Blow-up:} The blow-up of $X$ along the ideal sheaf $\Ish$ is defined as the relative Proj of the Rees algebra:
    \[ \tilde{X} := \Proj_X \left( \bigoplus_{d \ge 0} \Ish^d \right) ,\]
    \item \textbf{The Graph Closure:}  The graph closure $\Gamma$ of $\phi$ is defined as the scheme-theoretic image of the canonical morphism $U\to X \times \Prm^m_\Zrm$, according to \cite[Tag 01R7]{SP}.
\end{enumerate}
\end{definition}
Consider the relative projective space $X \times \Prm^m_\Zrm = \Proj_X(\Osh_X[y_0, \dots, y_m])$, where $y_0, \dots, y_m$ are the standard coordinates of $\Prm^m_\Zrm$. We can define a surjective homomorphism of graded $\Osh_X$-algebras
\[ \Osh_X[y_0, \dots, y_m] \twoheadrightarrow \bigoplus_{d \ge 0} (\Ish^d \otimes \Lsh^d )\]
which, in degree $1$, maps the free variable $y_i$ to the global section $s_i'\in H^0(X,\Ish\otimes\Lsh)$. 
By the functoriality of the $\Proj$ construction, it induces a closed imbedding
\[ \iota': \Proj_X \left( \bigoplus_{d \ge 0} (\Ish^d \otimes \Lsh^d )\right) \hookrightarrow \Proj_X(\Osh_X[y_0, \dots, y_m]) = X \times \Prm^m_\Zrm .\]

According to \cite[Lemma 7.9]{Har77}, the isomorphism of $X$-schemes
\[ \Proj_X \left( \bigoplus_{d \ge 0} (\Ish^d \otimes \Lsh^d )\right) \cong \Proj_X \left( \bigoplus_{d \ge 0} \Ish^d \right) = \tilde{X} \] induces a closed imbedding $\iota: \tilde{X} \hookrightarrow X \times \Prm^m_\Zrm$
such that the composite $\pi:\tilde{X}\xrightarrow{\iota}X \times \Prm^m_\Zrm \xrightarrow{pr_1}X$ is the blow-up morphism, while the composite $\Phi:\tilde{X}\xrightarrow{\iota}X \times \Prm^m_\Zrm \xrightarrow{pr_2}\Prm^m_\Zrm$ is compatible with $\phi$. Furthermore, we have $\Phi^*\Osh_{\Prm^m_\Zrm}(1)\cong \Osh_{\tilde{X}}(1)\otimes \pi^* \Lsh$.

Moreover, the factorization $U\xrightarrow{j}\tilde{X}\xrightarrow{\iota} X \times \Prm^m_\Zrm$ naturally induces a closed imbedding $\Gamma \hookrightarrow \tilde{X}$, which is an isomorphism if $X$ is integral: In the commutative diagram below, the right vertical morphism is injective. Therefore, the left vertical morphism is an isomorphism.
\[
\begin{tikzcd}
0 \arrow[r]
& \ker
  \arrow[r]
  \arrow[d, hook]
& \mathcal O_{X\times \Prm^m_\Zrm}
  \arrow[r]
  \arrow[d, equal]
& \iota_*\mathcal O_{\widetilde X}
  \arrow[r]
  \arrow[d]
& 0 \\
0 \arrow[r]
& \ker
  \arrow[r]
& \mathcal O_{X\times \Prm^m_\Zrm}
  \arrow[r]
& \iota_* j_*\mathcal O_U
\end{tikzcd}
\]

The arguments above yield the following more general statement.
\begin{proposition}\label{graph}
Let $X$ be an integral scheme, and let $\Lsh_1,\ldots,\Lsh_n$ be invertible sheaves on $X$. For each $i=1,\ldots,n$, let
\[
s_0^{(i)},s_1^{(i)},\ldots,s_{m_i}^{(i)}
\in H^0(X,\Lsh_i)
\]
be global sections that are not all identically zero, and let
$\Ish_i\subset\Osh_X$,
$\pi_i:\widetilde X_i\to X$,
$\iota_i:\widetilde X_i\hookrightarrow X\times\Prm^{m_i}_{\Zrm}$,
and
$\phi_i:U_i\to\Prm^{m_i}_{\Zrm}$
denote the associated ideal sheaf, blow-up, closed imbedding, and morphism constructed above.

Let $\pi:\widetilde X\to X$ be the blow-up of $X$ along the ideal sheaf $\Ish_1\cdots\Ish_n$, and let $\Gamma$ be the graph closure of
\[
\phi_1\times\cdots\times\phi_n:
U_1\cap\cdots\cap U_n
\longrightarrow
\Prm^{m_1}_{\Zrm}\times\cdots\times\Prm^{m_n}_{\Zrm}.
\]
Then there exists a closed imbedding
\[
\iota:\widetilde X
\hookrightarrow
X\times
\Prm^{m_1}_{\Zrm}\times\cdots\times\Prm^{m_n}_{\Zrm}
\]
which identifies $\Gamma$ with $\widetilde X$. Moreover, for each $i=1,\ldots,n$, there exists a morphism
$\eta_i:\widetilde X\to\widetilde X_i$
such that
\[
(pr_i\circ\iota)^*
\Osh_{\Prm^{m_i}_{\Zrm}}(1)
\cong
\eta_i^*\Osh_{\widetilde X_i}(1)\otimes\pi^*\Lsh_i, \]
\[\Osh_{\tilde{X}}(1)\cong \eta_1^*\Osh_{\widetilde X_1}(1)\otimes...\otimes\eta_n^*\Osh_{\widetilde X_n}(1).
\]
\end{proposition}
\begin{proof}
Observe that
\[
\widetilde X_1\times_X\cdots\times_X\widetilde X_n
\cong
\Proj_X\!\left(
\bigoplus_{d\ge0}
(\Ish_1^d\otimes\cdots\otimes\Ish_n^d)
\right).
\]
The natural surjective homomorphism of graded \(\Osh_X\)-algebras
\[
\bigoplus_{d\ge0}
(\Ish_1^d\otimes\cdots\otimes\Ish_n^d)
\twoheadrightarrow
\bigoplus_{d\ge0}
(\Ish_1\cdots\Ish_n)^d
\]
induces a closed imbedding
\[
\widetilde X
\hookrightarrow
\widetilde X_1\times_X\cdots\times_X\widetilde X_n.
\]
For each \(i\), the morphism \(\eta_i\) is obtained by composing it with the projection onto the \(i\)-th factor, while \(\iota\) is obtained by further composing with
\[
\iota_1\times\cdots\times\iota_n:
\widetilde X_1\times_X\cdots\times_X\widetilde X_n
\hookrightarrow
X\times
\Prm^{m_1}_{\Zrm}\times\cdots\times\Prm^{m_n}_{\Zrm}.
\]

Since \(U_1\cap\cdots\cap U_n\) is reduced, \(\Gamma\) is precisely the reduced induced closed subscheme structure on the closure of its image in
$
X\times
\Prm^{m_1}_{\Zrm}\times\cdots\times\Prm^{m_n}_{\Zrm}$.
As this closure is equal to \(\widetilde X\), it follows that
$\Gamma\cong\widetilde X$.
The commutative diagram below induces the claimed isomorphism of invertible sheaves
\[\begin{tikzcd}[column sep=large,row sep=large]
&& U_i \arrow[dr] & \\
U_1\cap...\cap U_n
\arrow[rru]
\arrow[r]
&
\widetilde X
\arrow[r, hook]
&
\widetilde X_1\times_X\cdots\times_X\widetilde X_n
\arrow[r,"pr_i"]
\arrow[d,hook,"\iota_1\times...\times\iota_n"]
&
\widetilde X_i
\arrow[d,hook,"\iota_i"]
\\
&&
\Prm^{m_1}_X\times_X\cdots\times_X\Prm^{m_n}_X
\arrow[r, "pr_i"]
&
X\times
\Prm^{m_i}_{\Zrm}
\arrow[r]
&
\Prm^{m_i}_{\Zrm}. 
\end{tikzcd}\]
\end{proof}

\subsection{Application to Multidegree Formula}
Let $k$ be a field. Applying Proposition~\ref{graph} to
\[
X=\Prm^{m+n+1}_k=\Pro \ k[x_0,\ldots,x_m,y_0,\ldots,y_n],
\qquad
\Lsh_1=\Lsh_2=\Osh_X(1),
\]
with sections
\[
x_0,\ldots,x_m\in H^0(X,\Lsh_1),
\qquad
y_0,\ldots,y_n\in H^0(X,\Lsh_2),
\]
we obtain ideal sheaves $\Ish_i\subset\Osh_X$, blow-up
$\pi:\widetilde X\to X$, open imbedding $j:U_1\cap U_2\to \tilde{X}$, closed imbedding
$
\iota:\widetilde X
\hookrightarrow
X\times\Prm^{m}_{k}\times\Prm^{n}_{k}$,
and morphisms $
\eta_i:\widetilde X\to\widetilde X_i$. We denote
\[
\begin{aligned}
pr &: X\times \mathbb P_k^m\times \mathbb P_k^n
   \to \mathbb P_k^m\times \mathbb P_k^n,\\
pr_1 &: X\times \mathbb P_k^m\times \mathbb P_k^n
     \to \mathbb P_k^m,\\
pr_2 &: X\times \mathbb P_k^m\times \mathbb P_k^n
     \to \mathbb P_k^n,\\
\xi& :=c_1(\pi^*\Osh_X(1)),\\
t_1& :=c_1(\iota^*pr_1^*\Osh_{\Prm^m_k}(1)),\\
t_2& :=c_1(\iota^*pr_2^*\Osh_{\Prm^n_k}(1)).
\end{aligned}
\]

Since $\widetilde X$ is identified with the graph closure, every fiber of
\[
pr\circ\iota:\widetilde X\to \mathbb P_k^m\times\mathbb P_k^n
\]
is isomorphic to $\mathbb P^1$. By Miracle Flatness, it follows that $pr\circ\iota$ is a smooth proper morphism of relative dimension~$1$.
Moreover, since $\eta_i^*\Osh_{\widetilde X_i}(-1)$ corresponds to the effective Cartier divisor $\widetilde X\times_X V(\Ish_i)$ on $\widetilde X$ and $V(\Ish_1)\cap V(\Ish_2)=\varnothing$, it follows that
\[
c_1\bigl(\eta_1^*\Osh_{\widetilde X_1}(-1)\bigr)\cdot
c_1\bigl(\eta_2^*\Osh_{\widetilde X_2}(-1)\bigr)=0,
\]
which is equivalent to
\[
\xi^2-(t_1+t_2)\xi+t_1t_2=0.
\]
Hence,
\[
\xi^\ell = h_{\ell-1}(t_1,t_2)\,\xi - t_1 t_2\, h_{\ell-2}(t_1,t_2)
\qquad (\ell \ge 1),
\]
where \(h_l(T_1,T_2)\in \mathbb Z[T_1,T_2]\) is defined by
\[
h_l(T_1,T_2)=
\begin{cases}
T_1^l + T_1^{l-1}T_2 + \cdots + T_2^l, & l \ge 0,\\
0, & l=-1.
\end{cases}
\]

Now let $V\subset \Prm^m_k\times\Prm^n_k$ be a closed subscheme. let $V'':=(pr\circ\iota)^{-1}(V)\subset \tilde{X}$
be its scheme-theoretic inverse image, and let $V'\subset X$
be the scheme-theoretic image of
$(pr\circ\iota\circ j)^{-1}(V)\to U_1\cap U_2 \to X$. It follows from \cite[Lemma 1.7.1]{Ful98} that
\[
\pi_*[V'']= [V'] \in Z_*(X).
\]
Furthermore, if
\[
[V]=\sum_i r_i[V_i]\in Z_*(\Prm^m_k\times\Prm^n_k)
\]
is the fundamental cycle of $V$, then
\[
[V'']=\sum_i r_i[V_i'']\in Z_*(\widetilde X),
\qquad
[V']=\sum_i r_i[V_i']\in Z_*(X).
\]
Consequently,
\[
\pi_*[V'']_\ell = [V']_\ell
\]
for every \(0\le \ell\le m+n+1\), where the subscript \(\ell\) denotes the \(\ell\)-dimensional component of a cycle.
\begin{remark}\label{conterex}
Note that if $V$ is an integral closed subscheme of dimension $d$, then $V'$ is integral of dimension $d+1$. Furthermore, the bihomogeneous equations defining $V\subset \Prm^m_k\times\Prm^n_k$ also define $V'$ as a closed subscheme of $\Prm^{m+n+1}_k$. Hence our construction recovers the situation treated by van der Waerden in \cite{vdW78}.

For non-integral subschemes, however, the relation between degrees and bidegrees is more subtle. For instance, let $V\subset \Prm^1_k\times\Prm^1_k$ be defined by
$x_0y_0=x_1y_1=0$.
Then the sum of the bidegrees of $V$ equals $2$, while the closed subscheme of $\Prm^3_k$ defined by the same equations has degree $4$.
\end{remark}

\begin{definition}
For integers $0\le a\le m+n+1$ and $0\le b\le m+n$, and for closed subschemes
$
Y\subset \Prm^{m+n+1}_k,
Z\subset \Prm^m_k\times\Prm^n_k,
$
we introduce the following notions of degree.
\begin{enumerate}
    \item \textbf{The codimension-$a$ degree of $Y$:}
    Under the identification
    \[
    A^*(\Prm^{m+n+1}_k)\cong \Zrm[T]/(T^{m+n+2}),
    \]
    the cycle class $[Y]$ is represented by a polynomial $f(T)$ of degree at most $m+n+1$. The codimension-$a$ degree of $Y$ is defined to be the coefficient of $T^a$ in $f(T)$.

    \item \textbf{The codimension-$b$ bidegrees of $Z$:}
    Under the identification
    \[
    A^*(\Prm^m_k\times\Prm^n_k)
    \cong
    \Zrm[T_1,T_2]/(T_1^{m+1},T_2^{n+1}),
    \]
    the cycle class $[Z]$ is represented by a polynomial $g(T_1,T_2)$ of total degree at most $m+n$. For each pair $(\alpha,\beta)$ satisfying $\alpha+\beta=b$, the coefficient of $T_1^\alpha T_2^\beta$ in $g(T_1,T_2)$ is called a codimension-$b$ bidegree of $Z$.
\end{enumerate}
\end{definition}
\begin{proposition}
Let $V\subset \Prm^m_k\times\Prm^n_k$ be a closed subscheme, and let $V''$ and $V'$ be the closed subschemes associated to $V$ by the above construction. Then, for every $0\le d\le m+n$, the codimension-$d$ degree of $V'$ is equal to the sum of the codimension-$d$ bidegrees of $V$.
\end{proposition}
\begin{proof}
For simplicity, we denote by $\tau$ the composition $pr\circ \iota$.
\begin{align*}
&\text{The codimension-}d\text{ degree of } V' \\
&= \int_X c_1(\mathcal{O}_X(1))^{m+n+1-d} \cap [V']_{m+n+1-d} \\
&= \int_X \pi_* \left( \xi^{m+n+1-d} \cap [V'']_{m+n+1-d} \right) \quad \text{(Projection formula)} \\
&= \int_{\tilde{X}} h_{m+n-d}(t_1, t_2) \xi \cap \left( \tau^* [V]_{m+n-d} \right) - t_1 t_2 h_{m+n-d-1}(t_1, t_2) \cap \left( \tau^* [V]_{m+n-d} \right) \\
&= \int_{\tilde{X}} \xi \cap \left[ \tau^* \left( h_{m+n-d}(c_1(\mathcal{O}_{\mathbb{P}^m_k}(1)|), c_1(\mathcal{O}_{\mathbb{P}^n_k}(1)|)) \cap [V]_{m+n-d} \right) \right] \\
&\quad - \tau^* \left( c_1(\mathcal{O}_{\mathbb{P}^m_k}(1)|) \cdot c_1(\mathcal{O}_{\mathbb{P}^n_k}(1)|) \cdot h_{m+n-d-1}(c_1(\mathcal{O}_{\mathbb{P}^m_k}(1)|), c_1(\mathcal{O}_{\mathbb{P}^n_k}(1)|)) \cap [V]_{m+n-d} \right),
\end{align*}
where $\mathcal{O}_{\mathbb{P}^m_k}(1)|$ denotes the pullback of $\mathcal{O}_{\mathbb{P}^m_k}(1)$ along the projection
$\mathbb{P}^m_k\times \mathbb{P}^n_k \to \mathbb{P}^m_k$, and analogously for $\mathcal{O}_{\mathbb{P}^n_k}(1)|$. Since all fibers of $\tau$ are isomorphic to $\Prm^1$, the first term in the last equation is the sum of the codimension-$d$ bidegrees of $V$, while the second term vanishes for dimensional reasons.
\end{proof}

\begin{remark}
One can generalize the above proposition to arbitrary multiprojective spaces: the codimension-$d$ degree of $V'$ is equal to the sum of the codimension-$d$ multidegrees of $V$.
\end{remark}

\bibliographystyle{amsalpha}

\bibliography{mybib}

\end{document}